\documentclass[12pt]{amsart}
\usepackage{amssymb,amsmath,amsthm,comment}
\usepackage{pdfpages,graphicx} 
\usepackage[section]{placeins} 
\usepackage{wrapfig}
\usepackage{float}
\usepackage{lineno}
\usepackage{youngtab}

\usepackage{setspace}
\newcommand{\shrinkmargins}[1]{
  \addtolength{\textheight}{#1\topmargin}
  \addtolength{\textheight}{#1\topmargin}
  \addtolength{\textwidth}{#1\oddsidemargin}
  \addtolength{\textwidth}{#1\evensidemargin}
  \addtolength{\topmargin}{-#1\topmargin}
  \addtolength{\oddsidemargin}{-#1\oddsidemargin}
  \addtolength{\evensidemargin}{-#1\evensidemargin}
  }
\shrinkmargins{0.5}

\newtheorem{theorem}{Theorem}
\newtheorem{lemma}[theorem]{Lemma}

\newtheorem{corollary}[theorem]{Corollary}
\newtheorem*{theorem*}{Theorem}
{Claim}

\theoremstyle{definition}

\theoremstyle{remark}

\numberwithin{theorem}{section} \numberwithin{equation}{section}

\usepackage{caption}
\usepackage{multirow}

\newcommand{\N}{\mathbb{N}}

\def\func#1{\mathop{\rm #1}}%

\begin{document}
\title{On a mod $3$ property of $\ell $-tuples of pairwise commuting permutations}
\author{Abdelmalek Abdesselam}
\address{Department of Mathematics, P. O. Box 400137, University of Virginia, Charlottesville, VA 22904-4137, USA}
\email{malek@virginia.edu}
\author{Bernhard Heim}
\address{Department of Mathematics and Computer Science\\Division of Mathematics\\University of Cologne\\ Weyertal 86-90 \\ 50931 Cologne \\Germany}
\address{Lehrstuhl A f\"{u}r Mathematik, RWTH Aachen University, 52056 Aachen, Germany}
\email{bheim@uni-koeln.de}
\author{Markus Neuhauser}
\address{Kutaisi International University, 5/7, Youth Avenue,  Kutaisi, 4600 Georgia}
\address{Lehrstuhl A f\"{u}r Mathematik, RWTH Aachen University, 52056 Aachen, Germany}
\email{markus.neuhauser@kiu.edu.ge}
\begin{abstract}
Let $S_n$ denote the symmetric group of permutations acting on $n$ elements.
We
investigate the
double sequence $\{N_{\ell}(n)\}$
counting the number
of $\ell$ tuples of elements of the symmetric group $S_n$, where the components commute, normalized by the order of $S_n$.
Our focus lies on exploring log-concavity with respect to $n$:
$$ N_{\ell}(n)^2 - N_{\ell}(n-1) \,\, N_{\ell}(n+1) \geq 0.$$
We
establish that this depends on $n \pmod{3}$ for sufficiently large $\ell$.
These numbers are studied by Bryan and Fulman as the $n$th orbifold
characteristics, generalizing work of Macdonald and
Hirzebruch--Hofer concerning the ordinary
and string-theoretic Euler characteristics of symmetric products.
Notably, $N_2(n)$ represents the partition numbers $p(n)$,
while $N_{3}(n)$ represents
the number of non-equivalent $n$-sheeted coverings of a torus studied by 
Liskovets and Medynkh. The numbers also appear in algebra since
$ \vert S_n \vert \,\, N_{\ell}(n) = \left\vert \func{Hom}\left( \mathbb{Z}^{\ell},S_n\right) \right\vert $.
\end{abstract}
\maketitle
\newpage
\section{Introduction and main results}
Let $S_n$ denote the symmetric group on $n$ elements.
This paper delves into
the log-concavity \cite{Br15} of commuting $\ell$-tuples of permutations in $S_n$.
Let
\begin{equation*}
C_{\ell,n}:= \Big\{ (\pi_1, \ldots, \pi_{\ell}) \in S_n^{\ell}\, : \,
\pi_j \pi_k = \pi_k \pi_j \text{ for } 1 \leq j,k \leq \ell \Big\}.
\end{equation*}
Let $n \geq 20$.
We
establish that the normalized number of elements
\begin{equation*}
N_{\ell}(n):= \frac{\vert C_{\ell,n} \vert}{\vert S_n\vert} \in \mathbb{N}
\end{equation*}
is log-concave with respect to $n$ at $n$, for almost all $\ell$ (i.~e.\ for all but finitely many), if
and only if $n$ is divisible by $3$.
This
double sequence $\{N_{\ell}(n)\}$
emerges as the specialization of the $\ell$th orbifold characteristic to the $n$th symmetric product of a manifold of ordinary Euler characteristic one,
as demonstrated in the work of Bryan and Fulman \cite{BF98}.
Let $g_{\ell}(n)$ be the number of subgroups of $\mathbb{Z}^{\ell}$ of index $n$.
Then:

\begin{theorem}[Bryan and Fulman, Theorem 2.1]
\label{th:BF}
	For $\ell \in \N$, we have
	\begin{equation*} 
		\sum_{n=0}^{\infty} P_n^{g_{\ell}}(x) \, t^n  =
		\prod_{n=1}^{\infty} \left( 1-t^n\right)^{- x \, g_{\ell-1}(n)} =
		\exp \left( x\, \sum_{n=1}^{\infty} g_{\ell}(n)  \frac{t^n}{n}\right),
	\end{equation*}
	where $P_n^{g_{\ell}}(x)$ is a polynomial of degree $n$ with $P_n^{g_{\ell}}(1)= N_{\ell}(n)$.
\end{theorem}
We also refer to Abdesselam et al.\ (\cite{ABDV23}, p. 3) and White \cite{Wh13}.

In the
instance where
$\ell=2$, we
discover that the polynomials correspond to
the D'Arcais polynomials \cite{DA13},
known in combinatorics as the (shifted) Nekrasov--Okounkov polynomials \cite{NO06, Ha10}. Further,
$P_{n}^{g_{\ell}}(1)=N_{2}(n)=p(n)$, where $p(n)$ 
denotes the number of partitions of $n$ (see \cite{An98, On04}).
Additionally, 
$N_{2}(n)$ equals
the number of conjugacy classes in $S_n$
\cite{ET68}.
Moving forward, $\ell =3$ gives a connection to topology 
(see the introduction of Britnell \cite{Br13}) due to work of
Liskovets and Medynkh \cite{LM09}. It
has been proven that $N_{3}(n)$ counts the number
of non-equivalent $n$-sheeted coverings of a torus.
Moreover, there is a direct connection
to algebra and geometry highlighted by White \cite{Wh13}. Obvious
bijections exist between the elements of $C_{\ell,n}$
and $\func{Hom} \left( \mathbb{Z}^{\ell}, S_n\right) $ and the number of actions of the group
$\mathbb{Z}^{\ell}$ on a set $X$ of $n$ labeled elements.

We call
the double sequence $\{N_{\ell}(n)\}$ log-concave at $n$ for $\ell$, if
\begin{equation*}\label{D}
\Delta_{\ell}(n):= N_{\ell}(n)^2 - N_{\ell}(n-1) \, N_{\ell}(n+1) \geq 0.
\end{equation*}
If $\Delta_{\ell}(n)<0$,
we
call the pair $(n,\ell)$ an exception (or strictly log-concave). In Section \ref{section1.1},
Table \ref{table1} and Table \ref{table2} we recorded all exceptions for 
\begin{equation*}
1 \leq n \leq 30 \text{ and } 1 \leq \ell \leq 40.
\end{equation*}
\newpage
\subsection{\label{section1.1}Landscape of exceptions}
\
\begin{table}[H]
\[
\begin{array}{rccccccccccccccccccccc}
\hline
n\backslash \ell &0&1&2&3&4&5&6&7&8&9&10&11&12&13&14&15&16&17&18&19&20\\ \hline \hline
1&&&\bullet &\bullet &\bullet &\bullet &\bullet &\bullet &\bullet &\bullet &\bullet &\bullet &\bullet &\bullet &\bullet &\bullet &\bullet &\bullet &\bullet &\bullet &\bullet \\
2&&&&&&&&&&&&&&&&&&&&&\\
3&&&\bullet &\bullet &\bullet &\bullet &\bullet &\bullet &\bullet &\bullet &\bullet &\bullet &\bullet &\bullet &&&&&&&\\
4&&&&&&&&&&&&&&&&&\bullet &\bullet &\bullet &\bullet &\bullet \\
5&&&\bullet &\bullet &\bullet &\bullet &\bullet &\bullet &\bullet &\bullet &\bullet &\bullet &\bullet &\bullet &\bullet &\bullet &\bullet &\bullet &\bullet &\bullet &\\
6&&&&&&&&&&&&&&&&&&&&&\\
7&&&\bullet &\bullet &\bullet &\bullet &\bullet &\bullet &\bullet &\bullet &\bullet &\bullet &\bullet &\bullet &\bullet &\bullet &\bullet &\bullet &\bullet &\bullet &\bullet \\
8&&&&&&&&&&&&&&&&\bullet &\bullet &\bullet &\bullet &\bullet &\bullet \\
9&&&\bullet &\bullet &\bullet &\bullet &\bullet &\bullet &\bullet &\bullet &\bullet &\bullet &\bullet &&&&&&&&\\
10&&&&&&&&&&&&&&&&&&&&&\\
11&&&\bullet &\bullet &\bullet &\bullet &\bullet &\bullet &\bullet &\bullet &\bullet &\bullet &\bullet &\bullet &\bullet &\bullet &\bullet &\bullet &\bullet &\bullet &\bullet \\
12&&&&&&&&&&&&&&&&&&&&&\\
13&&&\bullet &\bullet &\bullet &\bullet &\bullet &\bullet &\bullet &\bullet &&&&&&&&&&&\\
14&&&&&&&&&&&&&&&&&&&&&\bullet \\
15&&&\bullet &\bullet &\bullet &\bullet &\bullet &\bullet &\bullet &&&&&&&&&&&&\\
16&&&&&&&&&&&&&&&&&&&&&\\
17&&&\bullet &\bullet &\bullet &\bullet &\bullet &&&&&&&&&&&&&&\\
18&&&&&&&&&&&&&&&&&&&&&\\
19&&&\bullet &\bullet &\bullet &&&&&&&&&&&&&&&&\\
20&&&&&&&&&&&&&&&&&&&&&\\
21&&&\bullet &\bullet &&&&&&&&&&&&&&&&&\\
22&&&&&&&&&&&&&&&&&&&&&\\
23&&&\bullet &&&&&&&&&&&&&&&&&&\\
24&&&&&&&&&&&&&&&&&&&&&\\
25&&&\bullet &&&&&&&&&&&&&&&&&&\\
26&&&&&&&&&&&&&&&&&&&&&\\
27&&&&&&&&&&&&&&&&&&&&&\\
28&&&&&&&&&&&&&&&&&&&&&\\
29&&&&&&&&&&&&&&&&&&&&&\\
30&&&&&&&&&&&&&&&&&&&&&\\ \hline
\end{array}
\]
\caption{\label{table1}Exceptions for $N_{\ell}(n)$ for $0\leq \ell \leq 20$ and $1\leq n\leq 30$}
\end{table}
\newpage

\begin{table}[H]
\[
\begin{array}{rccccccccccccccccccccc}
\hline
n\backslash \ell &20&21&22&23&24&25&26&27&28&29&30&31&32&33&34&35&36&37&38&39&40\\ \hline \hline
1&\bullet &\bullet &\bullet &\bullet &\bullet &\bullet &\bullet &\bullet &\bullet &\bullet &\bullet &\bullet &\bullet &\bullet &\bullet &\bullet &\bullet &\bullet &\bullet &\bullet &\bullet \\
2&&&&&&&&&&&&&&&&&&&&&\\
3&&&&&&&&&&&&&&&&&&&&&\\
4&\bullet &\bullet &\bullet &\bullet &\bullet &\bullet &\bullet &\bullet &\bullet &\bullet &\bullet &\bullet &\bullet &\bullet &\bullet &\bullet &\bullet &\bullet &\bullet &\bullet &\bullet \\
5&&&&&&&&&&&&&&&&&&&&&\\
6&&&&&&&&&&&&&&&&&&&&&\\
7&\bullet &\bullet &\bullet &\bullet &\bullet &\bullet &\bullet &\bullet &\bullet &\bullet &\bullet &\bullet &\bullet &\bullet &\bullet &\bullet &\bullet &\bullet &\bullet &\bullet &\bullet \\
8&\bullet &\bullet &\bullet &\bullet &\bullet &\bullet &\bullet &&&&&&&&&&&&&&\\
9&&&&&&&&&&&&&&&&&&&&&\\
10&&&\bullet &\bullet &\bullet &\bullet &\bullet &\bullet &\bullet &\bullet &\bullet &\bullet &\bullet &\bullet &\bullet &\bullet &\bullet &\bullet &\bullet &\bullet &\bullet \\
11&\bullet &\bullet &\bullet &\bullet &\bullet &\bullet &\bullet &\bullet &\bullet &\bullet &\bullet &\bullet &\bullet &\bullet &\bullet &\bullet &&&&&\\
12&&&&&&&&&&&&&&&&&&&&&\\
13&&&&&&&\bullet &\bullet &\bullet &\bullet &\bullet &\bullet &\bullet &\bullet &\bullet &\bullet &\bullet &\bullet &\bullet &\bullet &\bullet \\
14&\bullet &\bullet &\bullet &\bullet &\bullet &\bullet &\bullet &\bullet &\bullet &\bullet &\bullet &\bullet &\bullet &\bullet &\bullet &\bullet &\bullet &\bullet &\bullet &\bullet &\bullet \\
15&&&&&&&&&&&&&&&&&&&&&\\
16&&&&&&&&&&\bullet &\bullet &\bullet &\bullet &\bullet &\bullet &\bullet &\bullet &\bullet &\bullet &\bullet &\bullet \\
17&&&&&\bullet &\bullet &\bullet &\bullet &\bullet &\bullet &\bullet &\bullet &\bullet &\bullet &\bullet &\bullet &\bullet &\bullet &\bullet &\bullet &\bullet \\
18&&&&&&&&&&&&&&&&&&&&&\\
19&&&&&&&&&&&&&\bullet &\bullet &\bullet &\bullet &\bullet &\bullet &\bullet &\bullet &\bullet \\
20&&&&&&&&\bullet &\bullet &\bullet &\bullet &\bullet &\bullet &\bullet &\bullet &\bullet &\bullet &\bullet &\bullet &\bullet &\bullet \\
21&&&&&&&&&&&&&&&&&&&&&\\
22&&&&&&&&&&&&&&&\bullet &\bullet &\bullet &\bullet &\bullet &\bullet &\bullet \\
23&&&&&&&&&&\bullet &\bullet &\bullet &\bullet &\bullet &\bullet &\bullet &\bullet &\bullet &\bullet &\bullet &\bullet \\
24&&&&&&&&&&&&&&&&&&&&&\\
25&&&&&&&&&&&&&&&&&\bullet &\bullet &\bullet &\bullet &\bullet \\
26&&&&&&&&&&&&&\bullet &\bullet &\bullet &\bullet &\bullet &\bullet &\bullet &\bullet &\bullet \\
27&&&&&&&&&&&&&&&&&&&&&\\
28&&&&&&&&&&&&&&&&&&&\bullet &\bullet &\bullet \\
29&&&&&&&&&&&&&&\bullet &\bullet &\bullet &\bullet &\bullet &\bullet &\bullet &\bullet \\
30&&&&&&&&&&&&&&&&&&&&&\\ \hline
\end{array}
\]
\caption{\label{table2}Exceptions for $N_{\ell}(n)$ for $20\leq \ell \leq 40$ and $1\leq n\leq 30$}
\end{table}

\subsection{Log-concavity of $\left\{ N_{\ell }\left( n\right) \right\} $ at $n$ for fixed $n$}
We state our main result.
\begin{theorem}\label{th:main}
Let $n\geq 2$. Then $\{N_{\ell}(n)\}$ is log-concave at $n$ for almost all
$\ell$, if $n \equiv 0 \pmod{3}$. Further, let $n\equiv 1 \pmod{3}$. Then
$\{N_{\ell}(n)\}$ is strictly log-convex at $n$ for almost all $\ell$. 
If $n\equiv 2 \pmod{3}$,
then $\{N_{\ell}(n)\}$ is
log-concave at $n$ for almost all $\ell$, if $n<20$ and strictly log-convex
at $n$ for allmost all $\ell $, if $n \geq 20$.
\end{theorem}

Through an in depth examination
of the proof, we obtain the following:

\begin{corollary}
\label{lineareschranken}The possible bounds
can be made explicit. Let
\begin{equation*}
\kappa
\left( n\right)
:= \left\{
\begin{array}{ll}
1+8\left( \frac{\left( \left( n/3\right) !\right) ^{2}p\left( n-1\right) p\left( n+1\right) }{3^{-2n/3}}-1\right) ,&n\equiv 0\pmod 3,\\
1+8\left( \left( \left( \frac{n-1}{3}\right) !p\left( n\right) \right) ^{2}2\cdot 3^{2\left( n-1\right) /3}-1\right) ,&n\equiv 1\pmod 3,\\
2,&n=2,\\
1+8\left( 71\cdot 2^{\left( n+1\right) /3}\left( \frac{n+1}{3}\right) !p\left( n+1\right) -1\right) ,&n\equiv 2\pmod 3,\\
&2<n\leq 17,\\
1+15\left( \frac{3!6!7!2^{10}}{211}\left( \left( \frac{1}{6!2^{6}}+p\left( 20\right) \right) ^{2}+\frac{p\left( 20\right) }{6!2^{5}}\right) -1\right) ,&n=20,\\
1+8\left( 97\cdot 2^{\left( n-2\right) /3}\cdot \left( \frac{n-2}{3}\right) !p\left( n\right) -1\right) ,&n\equiv 2\pmod 3,\\
&n\geq 23.
\end{array}
\right.
\end{equation*}
Then $\Delta_{\ell}(n) >0$ for $n\equiv 0 \pmod{3}$ and $ \ell \geq \kappa
(n)$,
and $\Delta_{\ell}(n) <0$ for $n\equiv 1 \pmod{3}$ and $ \ell \geq \kappa
(n)$.
Moreover, let $n \geq 20$ and $n \equiv 2 \pmod{3}$, then $\Delta_{\ell}(n) <0$ for
$\ell \geq \kappa \left( n\right) $.
\end{corollary}
This leads to the following result:
\begin{theorem}
Let $1\leq n \leq 20$.
Then we have the
complete classification of
exceptions as in Table \ref{klassifikation}.
\begin{table}[H]
\begin{tabular}{|r|c|r|c|}
\hline
$n$ & 
& n&\\ \hline
$1$ & $2 \leq \ell$ &$11$& $2 \leq \ell \leq 35$\\
$2$ & $\emptyset $ &$12$& $ \emptyset$\\
$3$ & $2\leq \ell \leq 13$ &$13$&
$2 \leq \ell \leq 9$, $26 \leq \ell $ \\
$4$ & $16 \leq \ell$ &$14$& $20 \leq \ell \leq 44$
\\
$5$ & $2\leq \ell \leq 19$ &$15$& $ 2 \leq \ell \leq 8$\\
$6$ & $\emptyset $ &$16$&  $ 29 \leq \ell$\\
$7$ & $2 \leq \ell$ &$17$&   $2 \leq \ell \leq 6$, $24 \leq \ell \leq 54$
\\
$8$ & $15\leq \ell \leq 26$ &$18$&  $\emptyset$ \\
$9$ & $2\leq \ell \leq 12$ &$19$&  $2 \leq \ell \leq 4$, $32 \leq \ell $   \\ 
$10$ & $ 22 \leq \ell $ &$20$&  $27 \leq \ell$       \\ \hline
\end{tabular}
\caption{\label{klassifikation}Classification of exceptions for $1 \leq n \leq 20$
}
\end{table}
\end{theorem}

\subsection{Log-concavity of $\left\{ N_{\ell}\left( n\right) \right\} $ for fixed $\ell$} \ \\
It is a well-known result
that the partition function $p(n)$ is log-concave, as established by
Nicolas \cite{Ni78}, and also by DeSalvo and Pak \cite{DP15}.
Consequently,
as $N_2(n)=p(n)$ the same log-concavity holds true for the sequence
$\{N_2(n)\}$.
Actually, $\Delta_2(n) >0$ for all even numbers $n$ and all odd numbers $n>25$.
Both proofs
rely on the Rademacher formula for partition numbers, involving the
Dedekind eta function.
Limited knowledge exists on $\ell >2$.
We refer to Table \ref{l=10} for the classification of all exceptions for $n \leq 10^5$
for $\ell=3$ and $n\leq 10^4$ otherwise.

\begin{table}[H]
\begin{tabular}{rlll}
\hline
$\ell $&log-concave&log-convex&proof/verification\\ \hline
$1$&$\mathbb{N}$&$\emptyset $&\\
$2$&$\{ n\in \mathbb{N}:n>25$ or $n$ even$\} $&$\{ n\in \mathbb{N}:1\leq n\leq 25$ odd$\}$&1978 Nicolas \\
$3$&$\{ n\in \mathbb{N}:n>21$ or $n$ even$\} $&$\{ n\in \mathbb{N}:1\leq n\leq 21$ odd$\} $&2024 H--N ($n\leq 10^{5}$)\\
$4$&$\{ n\in \mathbb{N}:n>19$ or $n$ even$\} $&$\{ n\in \mathbb{N}:1\leq n\leq 19$ odd$\} $&2024 H--N ($n\leq 10^{4}$)
\\
$5$&$\{ n\in \mathbb{N}:n>17$ or $n$ even$\} $&$\left\{
1,3,5,7,9,
11,13,15,17
\right\} $&$\vdots $\\
$6$&$\{ n\in \mathbb{N}:n>17$ or $n$ even$\} $&$\left\{
1,3,5,7,9,
11,13,15,17
\right\} $&$\vdots $\\
$7$&$\{ n\in \mathbb{N}:n>15$ or $n$ even$\} $&$\left\{ 1,3,5,7,9,11,13,15\right\} $&$\vdots $\\
$8$&$\{ n\in \mathbb{N}:n>15$ or $n$ even$\} $&$\left\{ 1,3,5,7,9,11,13,15\right\} $&$\vdots $\\
$9$&$\{ n\in \mathbb{N}:n>13$ or $n$ even$\} $&$\left\{ 1,3,5,7,9,11,13\right\} $&$\vdots $\\
$10$&$\{ n\in \mathbb{N}:n>11$ or $n$ even$\} $&$\left\{ 1,3,5,7,9,11\right\} $&2024 H--N ($n\leq 10^{4}$)\\ \hline
\end{tabular}
\caption{\label{l=10} Log-concavity patterns for $1\leq \ell \leq 10$}
\end{table}
Table \ref{l=10}
provides some evidence suggesting that for
a fixed $\ell$ the sequence $N_{\ell}(n)$ is log-concave for almost all $n$.
Additionally, we
refer to Table \ref{45}, where
we recorded the possible smallest $n_0=n_0(\ell)$ such that $N_{\ell}(n)$ is log-concave for all $n \geq n_0$.
\begin{table}[H]
\[
\begin{array}{rrrrrrrrrrr}
\hline
\ell _{1}\backslash \ell _{0}&1&2&3&4&5&6&7&8&9&10\\ \hline \hline
0&1&26&22&20&18&18&16&16&14&12\\
10&12&12&12&12&12&12&12&12&12&15\\
20&15&15&15&18&18&18&21&21&24&24\\
30&24&27&30&30&33&36&36&39&42&45\\ \hline
\end{array}
\]
\caption{\label{45}List of smallest $n_{0}$ such that $N_{\ell _{1}+\ell _{0}}\left( n\right) $ constitutes a log-concave sequence for all $n\geq n_{0}$ checked for $n\leq 10^{3}$}
\end{table}
\subsection{Exceptions for arbitrary $\ell$ large}
From our joint studies for large $\ell$ and $n$,
the following application arises.
\begin{corollary}
Let an arbitrary $N \in \mathbb{N}$ be given. Then there exists
a constant $L \in \mathbb{N}$, such that 
there exists an $n \geq N$ with
$\Delta _{\ell }\left( n\right) <0$ for all $\ell \geq L$.
\end{corollary}
This follows from Corollary \ref{lineareschranken}.
Let $n \geq 20$, $n \equiv 2 \pmod{3}$ and $n \geq N$. Let $ \ell \geq \kappa(n)$.
Then the result of Corollary \ref{lineareschranken} follows.
Further, we obtain:

\begin{corollary}
Let $\ell $ be given and $\left\{ N_{\ell }\left( n\right) \right\} $
be log-concave for almost all $n$.
Let $n_0(\ell)$ be given, such that $\Delta_{\ell}(n) >0$ for all $n \geq n_0(\ell)$.
Then the sequence $\{n_0(\ell)\}_{\ell}$ is unbounded.
\end{corollary}
\section{Preliminaries}
Let $n,\ell \geq 1$. Let $g_{\ell}(n)$ denote the number of subgroups of $\mathbb{Z}^{\ell}$ of index $n$. Then
\begin{equation*}
g_{\ell}(n)= \sum_{d
\mid n} d \, g_{\ell -1}(d),
\end{equation*}
where $g_0(n)=1$ for $n=1$ and $0$ otherwise. The function $g_{\ell}$ is
multiplicative and satisfies the following for $p$ prime and $m \in \mathbb{N}$:
\begin{equation*}
g_{\ell}(p^m) = \frac{ \left(p^{\ell}-1 \right) \cdots
\left(p^{\ell + m -1}-1 \right)}{\left( p -1\right) \cdots
\left(p^m -1 \right)}.
\end{equation*}
Therefore, $g_1(n)=g_{\ell}(1)=1$, $g_{\ell}(2)=2^{\ell}-1$, and $g_{\ell}(4)= \frac{1}{3} \left(2^{\ell}-1\right) \, \left( 2^{\ell +1}-1\right)$. 
Let $\ell \geq 2$. Then 
\begin{equation*}
n^{\ell -1} \leq g_{\ell}(n) \leq n^{\ell} \leq \sigma(n) \, n^{\ell -1},
\end{equation*}
where  $\sigma \left( n\right) := \sum_{d
\mid n} d$.

Moreover, the
identity (Theorem \ref{th:BF}):
\begin{equation*}
\sum_{n=0}^{\infty} N_{\ell}(n) \, t^n = \exp \left( \sum_{n=1}^{\infty} g_{\ell}(n) \, \frac{t^n}{n} \right)
\end{equation*}
implies that we have the following two properties:
\begin{eqnarray}
N_{\ell}(n) & =  & 
\sum _{k\leq n}\sum _{\substack{m_{1},\ldots ,m_{k}\geq 1 \\ m_{1}+\ldots +m_{k}=n}}
\frac{1}{k!} \, \frac{g _{\ell}\left( m_{1}\right) 
\cdots g _{\ell}\left( m_{k}\right) }{m_{1}\cdots m_{k}},
\\
N_{\ell}(n) & = & \frac{1}{n} \, \sum_{k=1}^n \, g_{\ell}(k) \, N_{\ell}(n-k), \text{ where } N_{\ell}(0):=1.
\end{eqnarray}
A straightforward calculation leads to:

\begin{lemma} Let $\ell, n \geq 1$, then 
$\Delta_{1}(n) = 0$. Let $\ell >1$, then  
$\Delta_{\ell}(1)<0 $ and 
$\Delta_{\ell}(2) >0$. Further, let $\ell >1$, then  $\Delta_{\ell}(3) <0$
for $2 \leq \ell \leq 13$ and $\Delta_{\ell}(3) >0$ for $\ell \geq 14$.
\end{lemma}

\section{Proof of Theorem \ref{th:main}}

Let $n,\ell \geq 2$. It is well-known \cite{BM65,MM65,SS20,HN22} that the
maximum $M_{1}\left( n\right) $
over all products of partitions of $n$ is given by
\begin{equation*}\label{max:M}
M_{1}\left( n\right) :=\max _{k\leq n}
\max _{\substack{m_{1},\ldots ,m_{k}\geq 1 \\ m_{1}+\ldots +m_{k}=n}}m_{1}\cdots m_{k}=\left\{
\begin{array}{ll}
3^{n/3},&n\equiv 0\pmod{3},\\
4\cdot 3^{\left( n-4\right) /3},&n\equiv 1\pmod{3},\\
2\cdot 3^{\left( n-2\right) /3},&n\equiv 2\pmod{3}.
\end{array}
\right.
\end{equation*}
This maximum is obtained for the following values of $k$
dependent
on $n \pmod{3}$:
\begin{equation*}
k=\left\{
\begin{array}{ll}
n/3, & n\equiv 0\pmod{3}, \\
\left( n-1\right) /3 \text{ or }\left( n+2\right) /3,
 & n\equiv 1\pmod{3}, \\
\left( n+1\right) /3, & n\equiv 2\pmod{3}.
\end{array}
\right.
\end{equation*}
Moreover, for each of these values of $k$,
the number of tuples reaching the maximum
is given by:
\begin{equation*}
\left\{
\begin{array}{ll}
1, & n\equiv 0\pmod{3}, \\
\frac{n-1}{3} \text{ or }\binom{\left( n+2\right) /3}{2},
 & n\equiv 1\pmod{3}, \\
\frac{n+1}{3}, & n\equiv 2\pmod{3}.
\end{array}
\right.
\end{equation*}
Since $n^{\ell -1} \leq g_{\ell}(n) \leq g_{2}(n) \,\, n^{\ell-1}$, we obtain:
\begin{equation}\label{estimates}
A_{\ell }\left( n
\right) \leq N_{\ell}(n) \leq N_{2}\left( n\right) \, \left( M_{1}\left( n\right) \right) ^{\ell -1}.
\end{equation}
The values of $A_{\ell }\left( n
\right) $ are given by
\begin{equation*}
A_{\ell }\left( n
\right) =\left\{
\begin{array}{ll}
\frac{3^{\left( \ell - 2\right) n/3}}{\left( n/3\right) !},
&n\equiv 0\pmod 3,\\
\frac{3\left( 4\cdot 3^{\left( n-4\right) /3}\right) ^{\ell -
2}}{2\left( \left( n-4\right) /3\right) !},
&n\equiv 1\pmod 3,\\
\frac{\left( 2\cdot 3^{\left( n-2\right) /3}\right) ^{\ell -2}}{\left( \left( n-2\right) /3\right) !},&n\equiv 2\pmod
3.
\end{array}
\right.
\end{equation*}
We recall that $N_2(n)= p(n)$.

\subsection{The case $n\equiv 0 \pmod{3}$}

We obtain
via (\ref{estimates}):
\begin{eqnarray*}
\frac{N_{\ell}(n)^2 }{N_{\ell}(n-1) N_{\ell}(n+1)} &\geq &
\frac{3^{-2n/3}
}{  \left((n/3)!\right)^2 \, p(n-1)\, p(n+1)}
\,\, \left(\frac{9}{8}     \right)^{\ell-1}\\
&>&\frac{3^{-2n/3}}{\left( \left( n/3\right) !\right) ^{2}p\left( n-1\right) p\left( n+1\right) }\left( 1+\frac{\ell -1}{8}\right) .
\end{eqnarray*}
Therefore, $\Delta_{\ell}(n)>0$ for
$\ell \geq 1+8\left( \frac{\left( \left( n/3\right) !\right) ^{2}p\left( n-1\right) p\left( n+1\right) }{3^{-2n/3}}-1\right) $.
\subsection{The case $n \equiv 1 \pmod{3}$}
We obtain
via (\ref{estimates}):
\begin{eqnarray*}
\frac{N_{\ell}(n)^2 }{N_{\ell}(n-1) N_{\ell}(n+1)} 
&\leq &
\left( \left( \frac{n-1}{3}\right) ! \,\,\,
p\left(
n
\right)
\right) ^{2} \,2\cdot 3^{2\left( n-1\right) /3}
\, \left(   \frac{8}{9}     \right)^{\ell-1}\\
&<&\left( \left( \frac{n-1}{3}\right) !p\left( n\right) \right) ^{2}2\cdot 3^{2\left( n-1\right) /3}\left( 1+\frac{\ell -1}{8}\right) ^{-1}.
\end{eqnarray*}
Therefore, $\Delta_{\ell}(n)<0$ for
$\ell \geq 1+8\left( \left( \left( \frac{n-1}{3}\right) !p\left( n\right) \right) ^{2}2\cdot 3^{2\left( n-1\right) /3}-1\right) $.
\subsection{The case $n \equiv 2 \pmod{3}$}
This particular case is more delicate
compared to all other cases.

A
direct calculation
reveals the following:

\begin{lemma}
For fixed $n>1$ the leading term of $N_{\ell }\left( n\right) $ in terms
of $\ell $ is obtained by
$B_{\ell }\left( n\right) =C_{1}\left( n\right) \left( M_{1}\left( n\right) \right) ^{\ell -1}$,
defined by
\[
C_{1}\left( n\right) =\left\{
\begin{array}{ll}
\frac{1}{2^{n/3}\left( \frac{n}{3}\right) !},&n\equiv 0\pmod 3,\\
\frac{7}{6}\frac{1}{2^{\left( n-4\right) /3}\left( \frac{n-4}{3}\right) !},&n\equiv 1\pmod 3,\\
\frac{1}{2^{\left( n-2\right) /3}\left( \frac{n-2}{3}\right) !},&n\equiv 2\pmod 3.
\end{array}
\right.
\]
The bases of the second order growth terms are obtained by
$M_{2}\left( n\right) $ defined as
\[
M_{2}\left( n\right) =\left\{
\begin{array}{ll}
2,&n=3,\\
2^{3}\cdot 3^{\left( n-6\right) /3
},&n>3,n\equiv 0\pmod 3
,\\
3,&n=4,\\
10,&n=7,\\
2^{5}\cdot 3^{\left( n-10\right) /3},&n>7,n\equiv 1\pmod 3
,\\
5,&n=5,\\
2^{4}\cdot 3^{\left( n-8\right) /3},&n>5,n\equiv 2\pmod 3
.\\
\end{array}
\right.
\]
Then we obtain
\[
B_{\ell }\left( n\right) \leq N_{\ell }\left( n\right) \leq
B_{\ell }\left( n\right) +p\left( n\right) M_{2}\left( n\right) .
\]
\end{lemma}

\begin{proof}
We obtain for $n\equiv 0\pmod{3}$ that the leading term
of
$N_{\ell }\left( n\right) $ in terms of $\ell $ is
$\frac{1}{2^{n/3}\left( \frac{n}{3}\right) !}\cdot 3^{
\left( \ell -1\right) n/3}
$.
For $n\equiv 1\pmod{3}$, $n\geq 4$, the
maximal growth term of
$N_{\ell }\left( n\right) $ with respect to $\ell $
is
$
2^{-\frac{n-4}{3}}
\frac{7
}{6\left( \left( n-4\right) /3\right) !}
\cdot \left( 4\cdot 3^{\frac{n-4}{3}}\right) ^{\ell -1}$.
For $n\equiv 2\pmod{3}$,
we obtain
$N_{\ell }\left( n\right) =
R+\frac{1}{\left( \left( n-2\right) /3\right) !}\cdot 2^{-\frac{n-2}{3}}\cdot \left( 2\cdot 3^{\frac{n-2}{3}}\right) ^{\ell -1}$,
where $R$ collects the terms of strictly slower growth. Similarly
for $M_{2}\left( n\right) $.
\end{proof}

Continuing
with the proof for $n\equiv 2\pmod 3$,
considering for
$n\equiv 2\pmod{3}$, $n\geq 5$, the terms of largest growth in $\ell $
are obtained as follows:
\begin{eqnarray*}
&&\frac{\left( N_{\ell }\left( n\right) \right) ^{2}}{N_{\ell }\left( n-1\right) N_{\ell }\left( n+1\right) }\\
&\geq &
\frac{\frac{\left( C_{1}\left( n\right) \right) ^{2}}{C_{1}\left( n-1\right) C_{1}\left( n+1\right) }\left( \frac{\left( M_{1}\left( n\right) \right) ^{2}}{M_{1}\left( n-1\right) M_{1}\left( n+1\right) }\right) ^{\ell -1}}%
{\left( 1+\frac{p\left( n-1\right) }{C_{1}\left( n-1\right) }\left( \frac{M_{2}\left( n-1\right) }{M_{1}\left( n-1\right) }\right) ^{\ell -1}\right) \left( 1+\frac{p\left( n+1\right) }{C_{1}\left( n+1\right) }\left( \frac{M_{2}\left( n+1\right) }{M_{1}\left( n+1\right) }\right) ^{\ell -1}\right) }\\
&\geq &\frac{6}{7}\frac{
n+1
}{
n-2
}\frac{1}{\left( 1+\frac{1}{71
}\right) ^{2}}>
\frac{6}{7}\frac{n+1}{n-2}\frac{35
}{36
}=\frac{5
}{6
}\frac{n+1}{n-2}\geq 1
\end{eqnarray*}
if $n\leq 17$ and
$\frac{p\left( n+s\right) }{C_{1}\left( n+s\right) }\left( \frac{M_{2}\left( n+s\right) }{M_{1}\left( n+s\right) }\right) ^{\ell -1}\leq \frac{1}{71}$
for $s=\pm 1$. This holds true if
$\ell \geq 1+\max _{s=\pm 1}\log _{M_{1}\left( n+s\right) /M_{2}\left( n+s\right) }\left( 71\frac{p\left( n+s\right) }{C_{1}\left( n+s\right) }\right) $.

Note that
\[
\frac{M_{1}\left( n-1\right) }{M_{2}\left( n-1\right) }=\left\{
\begin{array}{ll}
\frac{4}{3}\geq \frac{9}{8},&n=5,\\
\frac{12}{10}=\frac{6}{5}\geq \frac{9}{8},&n=8,\\
\frac{4\cdot 3^{\left( n-5\right) /3}}{2^{5}3^{\left( n-11\right) /3}}=\frac{9}{8},&n>8,n\equiv 2\pmod 3
\end{array}
\right.
\]
and
\[
\frac{M_{1}\left( n+1\right) }{M_{2}\left( n+1\right) }=
\frac{3^{\left( n+1\right) /3}}{2^{3}\cdot 3^{\left( n-
5\right) /3}}=\frac{9}{8}.
\]
Therefore,
\begin{eqnarray*}
\max _{s=\pm 1}\log _{M_{1}\left( n+s\right) /M_{2}\left( n+s\right) }\left( 71\frac{p\left( n+s\right) }{C_{1}\left( n+s\right) }\right) &\leq &\max _{s=\pm 1}\log _{9/8}\left( 71\frac{p\left( n+s\right) }{C_{1}\left( n+s\right) }\right) \\
&=&\log _{9/8}\left( 71
\cdot 2^{\left( n+1\right) /3}\left( \frac{n+1}{3}\right) !p\left( n+1\right) \right)
\end{eqnarray*}
since $C_{1}\left( n+1\right) <C_{1}\left( n-1\right) $
for $n\equiv 2\pmod 3$ and $n>2$.

In the opposite direction, we obtain
\begin{eqnarray*}
&&\left( N_{\ell }\left( n\right) \right) ^{2}-N_{\ell }\left( n-1\right) N_{\ell }\left( n+1\right) \\
&\leq &\frac{\left( C_{1}\left( n\right) \right) ^{2}}{C_{1}\left( n-1\right) C_{1}\left( n+1\right) }\left( \frac{\left( M_{1}\left( n\right) \right) ^{2}}{M_{1}\left( n-1\right) M_{1}\left( n+1\right) }\right) ^{\ell -1}\left( 1+\frac{p\left( n\right) }{C_{1}\left( n\right) }\left( \frac{M_{2}\left( n\right) }{M_{1}\left( n\right) }\right) ^{\ell -1}\right) ^{2}\\
&\leq &\frac{6}{7}\frac{n+1}{n-2}\left( 1+\frac{1}{9
7}\right) ^{2}<\frac{6}{7}\frac{n+1}{n-2}\frac{49}{48}=\frac{7}{8}\frac{n+1}{n-2}\leq 1
\end{eqnarray*}
if $n\geq 23$ and
$\frac{p\left( n\right) }{C_{1}\left( n\right) }\left( \frac{M_{1}\left( n\right) }{M_{2}\left( n\right) }\right) ^{\ell -1}\leq \frac{1}{97}$.
This holds true if
$\ell \geq 1+\log _{M_{1}\left( n\right) /M_{2}\left( n\right) }\left( 97\frac{p\left( n\right) }{C_{1}\left( n\right) }\right) =\log _{9/8}\left( 97\cdot 2^{\left( n-2\right) /3}\cdot \left( \frac{n-2}{3}\right) !p\left( n\right) \right) $
since
$
\frac{M_{1}\left( n\right) }{M_{2}\left( n\right) }=\frac{2\cdot 3^{\left( n-2\right) /3}}{16\cdot 3^{\left( n-8\right) /3}}=\frac{9}{8}
$.
In order to complete the analysis, let us consider the case $n=20
$. Including
the
third largest products as well,
we
derive the following through another straightforward calculation:

\begin{lemma}
In terms of the first three growth terms, we have
\[
N_{\ell }\left( 20\right) \leq C_{1}\left( 20\right) \left( M_{1}\left( 20\right) \right) ^{\ell -1}+C_{2}\left( 20\right) \left( M_{2}\left( 20\right) \right) ^{\ell -1}+p\left( 20\right) \left( M_{3}\left( 20\right) \right) ^{\ell -1}
\]
with
\[
C_{2}\left( 20\right) =\frac{
43}{4!
4!3!2^{3
}},\qquad M_{3}\left( 20\right) =5\cdot 3^{5}.
\]
With the first two growth terms, we
establish for $n\in \left\{ 19,21\right\} $
\[
N_{\ell }\left( n\right) \geq C_{1}\left( n\right) \left( M_{1}\left( n\right) \right) ^{\ell -1}+C_{2}\left( n\right) \left( M_{2}\left( n\right) \right) ^{\ell -1}
\]
with
\[
C_{2}\left( 19\right) =\frac{41}{6!2^{3}},\qquad C_{2}\left( 21\right) =\frac{1}{3!4!2^{5}}.
\]
\end{lemma}

\begin{proof}
For $n=21
\equiv 0
\pmod{3}$,
the leading terms of $N_{\ell }\left( n\right) $ regarding $\ell $
are
$\frac{1}{2^{
7}\cdot 7
!}\cdot 3^{7\ell
-7}+\frac{1
}{6\cdot 2^{
5}\cdot 4
!}\left( 8\cdot 3^{
5}\right) ^{\ell -1}$.
For $n=19\equiv 1\pmod{3}$,
we obtain the leading first two terms of
$N_{\ell }\left( n\right) $ regarding $\ell $ as
$\frac{
41}{6!
}\cdot 2^{5\ell -8
}\cdot 3^{3\ell -3}+\frac{7}{2^{5}\cdot 6!}\cdot \left( 4\cdot 3^{5}\right) ^{\ell -1}$.
For $n=20\equiv 2
\pmod{3}$ we obtain the first two leading terms of
$N_{\ell }\left( n\right) $ regarding $\ell $ as
$2^{\ell -
7}\cdot 3^{6\ell -6}\cdot \frac{1}{6!}+2^{4\ell -
8}\cdot 3^{4\ell -4}\cdot \frac{43}{4!4!3}$.

For $M_{3}\left( 20\right) $ the calculation is similar.
\end{proof}

Therefore, regarding the proof for $n=20$, we establish
\begin{eqnarray*}
&&\left( N_{\ell }\left( 20\right) \right) ^{2}-N_{\ell }\left( 19\right) N_{\ell }\left( 21\right)
\\
&\leq &\left( C_{1}\left( 20\right) \left( M_{1}\left( 20\right) \right) ^{\ell -1}+C_{2}\left( 20\right) \left( M_{2}\left( 20\right) \right) ^{\ell -1}+p\left( 20\right) \left( M_{3}\left( 20\right) \right) ^{\ell -1}\right) ^{2}\\
&&{}-\left( C_{1}\left( 19\right) \left( M_{1}\left( 19\right) \right) ^{\ell -1}+C_{2}\left( 19\right) \left( M_{2}\left( 19\right) \right) ^{\ell -1}\right) \\
&&\phantom{{}-}\cdot \left( C_{1}\left( 21\right) \left( M_{1}\left( 21\right) \right) ^{\ell -1}+C_{2}\left( 21\right) \left( M_{2}\left( 21\right) \right) ^{\ell -1}\right) \\
&<
&-2^{5\ell -
15}\cdot 3^{10\ell -10}\cdot \frac{211}{7!3!6!}+\left( C_{2}\left( 20\right) \left( M_{2}\left( 20\right) \right) ^{\ell -1}
+
p\left( 20\right) \left( M_{3}\left( 20\right) \right) ^{\ell -1}\right) ^{2}\\
&&{}+C_{1}\left( 20\right) p\left( 20\right) \left( M_{1}\left( 20\right) M_{3}\left( 20\right) \right) ^{\ell -1}\\
&<&-\frac{211}{7!3!6!2^{10}}\left( 2^{5}3^{10}\right) ^{\ell -1}+\left( \left( C_{2}\left( 20\right) +p\left( 20\right) \right) ^{2}+2C_{1}\left( 20\right) p\left( 20\right) \right) \left( 10\cdot 3^{11}\right) ^{\ell -1}\leq 0
\end{eqnarray*}
if
\begin{eqnarray*}
\ell &\geq &1+\log _{16/15}\left( \frac{2^{10}\cdot 3!6!7!\left( \left( C_{2}\left( 20\right)
+p\left( 20\right) \right) ^{2}+2C_{1}\left( 20\right) p\left( 20\right) \right) }{211}\right) \\
&=&1+\log _{16/15}\left( \frac{3!6!7!2^{10}}{211}\left( \left( \frac{
43}{4!4!
3!2^{3
}}+p\left( 20\right) \right) ^{2}+\frac{p\left( 20\right) }{6!2^{5}}\right) \right)
\end{eqnarray*}
since
$\left( M_{2}\left( 20\right) \right) ^{2}=2^{8}3^{8}<10\cdot 3^{11}=M_{3}\left( 20\right) M_{1}\left( 20\right) $.

Recall that $N_{\ell }\left( 1\right) =1$ and
$N_{\ell }\left( 2\right) =2^{\ell -1}$ and we obtain
$N_{\ell }\left( 3\right) =
\frac{1}{2}\cdot 3^{\ell -1}+2^{\ell -1}-\frac{1}{2}$.
Therefore,
$\frac{N_{\ell }\left( 1\right) N_{\ell }\left( 3\right) }{\left( N_{\ell }\left( 2\right) \right) ^{2}}
>1$
for $\ell \geq 2$ and it is $\frac{2}{3}$
for $\ell =0$ and $1$ for $\ell =1$.
\subsection{\label{bessereschranken}Remark}
Note that from the proof
we can actually derive the following bounds
for $\ell \geq L\left( n\right) $:
\begin{equation*}
L \left( n\right) =\left\{
\begin{array}{ll}
1+
\log _{9/8}\left( \left( \left( n/3\right) !\right) ^{2}
p\left( n-1\right)
p\left( n+1\right)
3^{3n/2}\right) ,&n\equiv 0\pmod 3,\\
1+
\log _{9/8}\left( \left( \left( \frac{n-1}{3}\right) !p\left( n\right) \right) ^{2}
\cdot 2\cdot
3^{2\left( n-1\right) /3}\right) ,&n\equiv 1\pmod 3,\\
2,&n=2,\\
1+\log _{9/8}\left( 71\cdot 2^{\left( n+1\right) /3}\left( \frac{n+1}{3}\right) !p\left( n+1\right) \right) ,&n\equiv 2\pmod 3,\\
&2<n\leq 17,\\
1+\log _{16/15}\left( \frac{3!6!7!2^{10}}{211}\left( \left( \frac{1}{6!2^{6}}+p\left( 20\right) \right) ^{2}+\frac{p\left( 20\right) }{6!2^{5}}\right) \right) ,&n=20,\\
1+\log _{9/8}\left( 97\cdot 2^{\left( n-2\right) /3}\cdot \left( \frac{n-2}{3}\right) !p\left( n\right) \right) ,&n\equiv 2\pmod 3,\\
&n\geq 23.
\end{array}
\right.
\label{eq:bessereschranken}
\end{equation*}
For $1\leq n\leq 20$ this leads to the bounds in Table~\ref{schrankenn}.
\begin{table}[H]
\[
\begin{array}{rrrr}
\hline
n&\left\lceil L\left( n\right) \right\rceil &n&\left\lceil L\left( n\right) \right\rceil \\ \hline
1&7&11&125\\
2&2&12&203\\
3&40&13&214\\
4&53&14&152\\
5&76&15&264\\
6&90&16&274\\
7&102&17&179\\
8&99&18&326\\
9&146&19&336\\
10&157&20&487\\ \hline
\end{array}
\]
\caption{\label{schrankenn}Values of $\left\lceil L\left( n\right) \right\rceil $}
\end{table}

The proof of Corollary~\ref{lineareschranken} and
Table~\ref{klassifikation} is now complete, after
verifying the
remaing cases with PARI/GP.

\section{Data
availability
statement}

The datasets generated during and/or analysed during the current study are available from the corresponding author on reasonable request.

\section{Ethics
declarations}

\subsection{Conflict of
interest}

The authors declare that they have no conflict of interest.
\newline
\newline
{\bf Acknowledgments.}
The authors
express their gratitude
to Ken Ono for bringing together
their
shared research interests.

\end{document}